\documentclass[twoside, 
]{amsart}

  \usepackage[utf8]{inputenc}
  \usepackage{amsmath}
\usepackage{enumerate}
  \usepackage{amssymb}
  \usepackage{amsthm}
  \usepackage{appendix}
  \usepackage{bm}
  \usepackage{graphicx}
  \usepackage{color}
\usepackage{accents}
\usepackage{todonotes}

\newtheorem{Th}{Theorem}
\newtheorem{Prop}{Proposition}[section]
\newtheorem{lem}{Lemma}
\newtheorem{factor}{Factor}
\newtheorem{rem}{Remark}

\newtheorem{Cor}{Corollary}
\newtheorem*{conj}{Conjecture}

\newtheorem*{que}{Question}
\newtheorem*{thank}{\ \ \ Acknowledgment}

\def\scalar(#1,#2){(#1\mid#2)}

\renewcommand{\hat}{\widehat}

\newcommand{\ca}{\mathcal{A}}
\newcommand{\cb}{\mathcal{B}}

\newcommand{\bmu}{\bm \mu}

\newcommand{\R}{{\mathbb{R}}}
\newcommand{\Pro}{{\mathbb{P}}}
\newcommand{\T}{{\mathbb{T}}}
\newcommand{\C}{{\mathbb{C}}}
\newcommand{\Z}{{\mathbb{Z}}}
\newcommand{\N}{{\mathbb{N}}}
\newcommand{\E}{{\mathbb{E}}}

\newcommand{\mob}{\boldsymbol{\mu}}
\newcommand{\lamob}{\boldsymbol{\lambda}}
\newcommand{\bnu}{\boldsymbol{\nu}}
\newcommand{\vMan}{\boldsymbol{\Lambda}}
\newcommand{\tend}[3][]{\xrightarrow[#2\to#3]{#1}}

\newcommand{\ds}{\displaystyle}

\newcommand*{\Resize}[2]{\resizebox{#1}{!}{$#2$}}
\newcommand*{\Sup}{\Resize{0.7cm}{sup}}
\newcommand{\TODOH}[1]{{#1}}

\newcommand*{\QEDA}{\hfill\ensuremath{\blacksquare}}

\title[A cubic nonconventional ergodic average]{A cubic nonconventional ergodic average with multiplicative or Mangoldt weights}
\author{E. H. El Abdalaoui}
\address{University
of Rouen Normandy, Department of Mathematics, LMRS, UMR 60 85, Avenue de l'Universit\'e, BP.12, 76801
Saint Etienne du Rouvray - France}
\email{elhoucein.elabdalaoui@univ-rouen.fr }
\author{Xiangdong Ye}
\address{Wu Wen-Tsun Key Laboratory of Mathematics, USTC, Chinese Academy of Sciences, Department of Mathematics, University
of science and technology of China, Hefei, Anhui, 230026- China}
\thanks{The second author is supported by NNSF of China (11371339  and 11431012).\\
}

\email{yexd@ustc.edu.cn}

\begin{document}
\date{\today}

{\renewcommand\abstractname{Abstract}
\begin{abstract}
We show that the cubic nonconventional ergodic averages of any order with a bounded multiplicative function
weight converge almost surely to zero provided that the multiplicative function satisfies a strong
Daboussi-Delange condition. We further obtain that the Ces\`{a}ro mean of the self-correlations
and some moving average of the self-correlations of such multiplicative functions converge to zero. Our proof
gives, for any $N \geq 2$,
$$\frac1{N}\sum_{m=1}^{N}\Big|\frac1{N}\sum_{n=1}^{N} \bnu(n) \bnu(n+m)\Big| \leq \frac{C}{\log(N)^{\epsilon}},$$
and
$$\frac1{N^2}\sum_{n,p=1}^{N}\Big|\frac1{N}\sum_{m=1}^{N} \bnu(m) \bnu(n+m)\bnu(m+p)\bnu(n+m+p)\Big| \leq \frac{C}{\log(N)^{\varepsilon}},$$
where $C,\varepsilon$ are some positive constants and $\bnu$ is a bounded multiplicative function
satisfying a Daboussi-Delange condition with logarithmic speed.  We further establish that the cubic
nonconventional ergodic averages of any order with Mangoldt weight converge almost surely provided that all the systems are
nilsystems.





\end{abstract}

\maketitle
\section{Introduction.}
The purpose of this note is motivated, on one hand, by the recent great interest on the M\"{o}bius
function from the dynamical point view, and on the other hand, by the problem of the multiple recurrence which
goes back to the seminal work of Furstenberg \cite{Fur}. This later problem has nowadays a long history.\\

The dynamical study of M\"{o}bius function was initiated recently by Sarnak in \cite{Sarnak4}.
There, Sarnak made a conjecture that  the M\"{o}bius
function is orthogonal to any deterministic dynamical sequence. Notice that this precise the definition of a reasonable
sequence in the  M\"{o}bius randomness law mentioned by Iwaniec-Kowalski in \cite[p.338]{Iwaniec}.  Sarnak further
mentioned that Bourgain's approach allows to prove that for almost all point $x$ in any measurable dynamical system
$(X,\ca,T,\Pro)$, the M\"{o}bius function is orthogonal to any dynamical sequence $f(T^nx)$. In \cite{al-lem},
using a spectral theorem combined with Davenport' estimation and Etmadi's trick, the authors gave a simple proof of this fact. 
Subsequently, Cuny and Weber gave a  proof in which they showed  that there is a rate in this almost sure
convergence \cite{Cuny-W}. They further used Bourgain's method to prove that the almost sure convergence holds
for other arithmetical functions, like the divisor function, the theta function and  the generalized
Euler totient function. Very recently, using Green-Tao estimation \cite{Green-Tao} combined with the method in \cite{al-lem}, Eisner in \cite{Tanja}
proved that almost surely the dynamical sequence $f(T^{p(n)}x)$, where $p$ is an integer polynomial,
is orthogonal to the M\"{o}bius function. She further proved that the M\"{o}bius function is a good
weight (with limit $0$) for the multiple polynomial mean ergodic theorem by using Qing Chu's result \cite{Chu}.
Subsequently, in a very recent preprint \cite{Nikos-Host},
Host and Frantzikinakis  established that any multiplicative function with mean value along any arithmetic sequence
is a good weight for the multiple polynomial mean ergodic theorem (with limit $0$ for aperiodic multiplicative functions). \\

Here, we are interested in the pointwise convergence of cubic nonconventional ergodic averages with  multiplicative functions
satisfying a strong Daboussi-Delange condition. \\

The convergence of cubic nonconventional ergodic averages was initiated by  Bergelson in \cite{Berg},
where convergence in $L^2$ was shown for order 2 and under the extra assumption that all the
transformations are equal. Under the same assumption, Bergelson's result was extended by
Host and Kra for cubic averages of order 3  in \cite{Host-K1}, and for arbitrary order in \cite{Host-K2}.
Assani proved that pointwise convergence of cubic nonconventional ergodic averages of order 3 holds for not necessarily
commuting maps in \cite{Assani1, Assani2}, and further established the pointwise convergence for cubic averages of
arbitrary order when all the maps are equal.
In \cite{chu-nikos}, Chu and Frantzikinakis completed the study and established the pointwise convergence for
the cubic averages of arbitrary order.
Very recently, Huang-Shao and Ye \cite{Ye} gave a topological-like proof of the pointwise convergence
of the cubic nonconventional ergodic average when all the maps are equal. They further applied their
method to obtain the pointwise convergence of nonconventional ergodic averages for a distal system.\\

Here, we establish that the cubic averages of any order with the multiplicative function
weight satisfying a strong Daboussi-Delange condition converge to zero almost surely.  The proof depends
heavily on the double recurrence Bourgain's theorem (DRBT for short) \cite{BourgainD}. As a consequence,
we obtain that the cubic averages of any order with M\"{o}bius or Liouville weights converge to zero
almost surely.
We further obtain an estimation of the speed of the convergence of the C\'esaro mean of the self-correlation of the
multiplicative function satisfying a Daboussi-Delange condition with logarithmic speed. Those C\'esaro means of the self-correlation
converge obviously to zero.
Our proofs yields that some moving averages of the self-correlation of the multiplicative function satisfying a
Daboussi-Delange condition with logarithmic speed are summable along any divergent geometric sequence.\\

We further obtain with the help of Davenport estimation that some moving averages of the
self-correlation of  M\"{o}bius function or Liouville function  converge to zero, and they
are summable along any divergent geometric sequence. Moreover, We establish that  the cubic nonconventional ergodic averages weighted with 
von-Mangoldt function converge provided that all the systems are nilsystems. This result is obtained as a consequence of 
the very recent result of Ford-Green-Konyagin-Tao \cite{FGKtao} combined with some ingredients from \cite{Assani1}.
\vskip 0.2cm

The paper is organized as follows. In section 2, we recall the main ingredients needed for the proof. In section 3, we state
our main results and its consequences, and prove our first main result. In section 4, we give proofs of our other main results.

When this paper was under preparation, we learned that Matom\"{a}ki, Radzwi{\l}{\l} and Tao \cite{MRtao} proved that for any natural number
$k$, and for any $10 \leq H \leq X$, we have
$$\sum_{1 \leq h_1,h_2,\cdots,h_k \leq H}\Big|\sum_{1 \leq n \leq X}\lamob(n+h_1)\cdots
\lamob(n+h_k)\Big|\ll k \Big(\frac{\log\log H}{\log H}+\frac1{\log^{\frac1{3000}}X}\Big) H^{k-1} X.$$
In the case $k=2$, this gives
$$\sum_{1 \leq h \leq X} \Big|\sum_{1 \leq n \leq X}\lamob(n)\lamob(n+h)\Big|\ll k \Big(\frac{\log\log H}{\log H}+\frac1{\log^{\frac1{3000}}X}\Big) H X.$$
This last estimation is largely bigger than our estimation when $H=X.$\\

We remind that besides, some estimation of limsup and liminf of some correlations of Liouville were obtained by several
authors: Graham and Hansely, Harman-Pintz and Wolke, and Cassaigne-Ferenczi-Mauduit-Rivat and S\'{a}rk\"{o}zy.
We refer to \cite{CF} for more details and for the references on the subject. We further refer to the recent
work of Matom\"{a}ki, Radzwi{\l}{\l} \cite{MR} in which the authors proved that for any $h \geq 1$, there exits $\delta(h)>0$ such that
$$\frac1{X}\Big|\sum_{j=1}^{X} \lamob(j) \lamob(j+1)\Big|<1-\delta(h),$$
for all large enough $X>1$.

\section{Basic definitions and tools.}\label{Sec:dt}
Recall that the Liouville function $\lamob ~~:~~ \N^*\longrightarrow  \{-1,1\}$ is defined by
$$
\lamob(n)=(-1)^{\Omega(n)},
$$
where $\Omega(n)$ is the number of prime factors of $n$ counted with multiplicities with $\Omega(1)=1$. \TODOH{Obviously,
$\lamob$ is completely mutiplicative, that is, $\lamob(nm)=\lamob(n)\lamob(m),$ for any $n,m \in \N^*$}. The integer $n$ is said to be not square-free if there is a prime number $p$ such that $n$ is in the class of $0$ mod $p^2$.
The M\"{o}bius function
$\bmu ~:~ \N \longrightarrow \{-1,0,1\}$ is define as follows
$$
\bmu(n)=\begin{cases}
\lamob(n),& \text{ if $n$ is square-free ;}\\
1,& \text{ if }n=1;\\
0,& \text{ otherwise}.
\end{cases}
$$
This definition of M\"{o}bius function establishes that the the restriction of Liouville function and M\"{o}bius function to set of
square free numbers coincident. Nevertheless, the  M\"{o}bius function is only mutiplicative, that is,
$\bmu(mn)=\bmu(m)\bmu(n)$ whenever $n$ and $m$ are coprime.\\

We further remind that the von Mangoldt function is given by
$$
\vMan(n)=\begin{cases}
\log(p),& \text{ if $n=p^\alpha,$ for some prime $p$ and $\alpha \geq 1$ ;}\\
0,& \text{ otherwise}.
\end{cases}
$$
We shall need the following crucial result due to Davenport \cite{Da} 
\begin{Prop}\label{PDAV1}For any $\epsilon>0$, for any $N \geq 2$, we have
\begin{eqnarray}\label{Dav1}
\sum_{n=1}^{N}\bmu(n) e^{2\pi i n  t} = O\Big(\frac{N}{\ln(N)^{\epsilon}}\Big),
\end{eqnarray}
uniformly in $t$.
\end{Prop}
\noindent{}By Batman-Chowla's argument in \cite{Bat-Cho} we have\footnote{See the proof of Lemma 1 in \cite{Bat-Cho}.}
\begin{Prop}\label{PDAV2}
For any $\epsilon>0$, for any $N \geq 2$, we have
\begin{eqnarray}\label{Dav2}
\sum_{n=1}^{N}\lamob(n) e^{2\pi i n t} = O\Big(\frac{N}{\ln(N)^{\epsilon}}\Big),
\end{eqnarray}
uniformly in $t$.
\end{Prop}
The previous estimation seems to be the best known result. However, if we restrict our self to the subset of the circle and allowed the estimation
non to be uniform, Murty and Sankaranarayanan \cite{MSankara} proved the following

\begin{Prop}\label{MSanka}
For any $\epsilon>0$, for any $N \geq 2$, we have
\begin{eqnarray}\label{MS}
\sum_{n=1}^{N}\lamob(n) e^{2\pi i n t} = O\Big(N^{{\frac{4}{5}}+\epsilon}\Big),
\end{eqnarray}
for all $t$ of type 1.
\end{Prop}
We remind that the irrational number $t$ is said to be of type $\eta$ if $\eta$ is the supremum of all $\gamma$ for which
$\liminf_{\overset{q\longrightarrow +\infty}{q \in \N}}q^{\gamma}\|qt\|_{\Z}=0,$
where, as is customary, $\|x\|_{\Z}$ denotes the distance to the nearest integer.\\

The authors in \cite{MSankara} pointed out that the following estimation is implicit in \cite{Da}. That is,
for all $t$ of type 1,
\begin{eqnarray}\label{Man}
\sum_{n=1}^{N}\vMan(n) e^{2\pi i n t} = O\Big(N^{{\frac{4}{5}}+\epsilon}\Big),
\end{eqnarray}
Here, obviously the constant in the estimation depend on $\alpha$. At this point, one may ask if Proposition \ref{MSanka} can
be improved by establishing that the estimation is uniform. If the answer is yes then our result (Theorem \ref{Main-Speed1})
can be much more improved. \TODOH{ But, following the state of art in the subject, we can define the following class of mutiplicative functions.\\

The multiplicative function is said to satisfy the strong Daboussi-Delange condition with logarithmic speed if for any $N \geq 2$, we have
$$\Big|\frac1{N}\sum_{n=1}^{N}\bnu(n)e^{2i\pi n \alpha}\Big| \leq \frac{c}{\log(N)^{\kappa}},~~~~\textrm{for~~all~~} \alpha ,$$
with $c,\kappa$ are positive constants, $\kappa<1$. \\

By Propositions \ref{PDAV1}, \ref{PDAV2}, the M\"{o}bius and Liouville functions satisfy the Daboussi-Delange condition
with logarithmic speed.
We further say that the multiplicative function $\bnu$ satisfies the strong Daboussi-Delange condition if
$$\sup_{ \alpha \in \R}\Big|\frac1{N}\sum_{n=1}^{N}\bnu(n)e^{2i\pi n \alpha}\Big| \tend{N}{+\infty}0.$$}

The  M\"{o}bius and Liouville functions are connected to the Riemann zeta function by the following
$$
\frac1{\zeta(s)}=\sum_{n=1}^{\infty}\frac{\mob(n)}{n^s}
\text{ and } \frac{\zeta(2s)}{\zeta(s)}=\sum_{n=1}^{\infty}\frac{\lamob(n)}{n^s}
\text{ for any }s\in\mathbb{C}\text{ with }\Re(s)>1.
$$
Let us recall that Chowla made a conjecture in~\cite{Cho} on the multiple self-correlation of $\bmu$ which can be stated as follows:
\begin{conj}[of Chowla]
For each choice of $1\leq a_1<\dots<a_r$, $r\geq 0$, with $i_s\in \{1,2\}$, not all equal to~$2$, we have
\begin{equation}\label{cza}
\sum_{n\leq N}\mob^{i_0}(n)\cdot \mob^{i_1}(n+a_1)\cdot\ldots\cdot\mob^{i_r}(n+a_r)={\rm o}(N).
\end{equation}
\end{conj}
In \cite{Sarnak4}, Sarnak noticed  that Chowla conjecture is a notorious conjecture in number theory, and
formulated the following conjecture:
\begin{conj}[of Sarnak]
For any dynamical system $(X,T)$, where $X$ is a compact metric space and $T$ is a homeomorphism of zero
topological entropy, for any $f\in C(X)$ and any $x \in X$, we have
\begin{equation}\label{sar}
\sum_{n\leq N}f(T^nx)\mob(n)={\rm o}(N).
\end{equation}
\end{conj}
He further announced that Chowla conjecture implies Sarnak conjecture, and wrote ``we persist in maintaining
Conjecture \eqref{sar} as the central one even though it is much weaker than Conjecture \eqref{cza}. The point
is that Conjecture \eqref{sar} refers only to correlations of $\bmu$ with deterministic sequences and avoids the
difficulties associated with self-correlations." For the ergodic proof of the fact that Chowla conjecture implies
Sarnak conjecture we refer the reader to \cite{al-lem} and the references therein.\\

\TODOH{At this point, let us mention that in many cases Sarnak's conjecture holds as a consequence of the following criterion.
\begin{Prop}[Katai-Bourgain-Sarnak-Ziegler criterion]\label{KBSZ}
Let $(X,\ca,\mu)$ be a Lebesgue probability space and
$T$ be an invertible measure preserving transformation. Let $\bnu$ be a multiplicative,
$f \in L^{\infty}$ with $\|f\|_{\infty} \leq 1$ and
$\varepsilon>0$.  Assume that for almost all point $x\in X$ and for all different prime numbers
$p$ and $q$ less than $\exp(1/\varepsilon)$, we have
  \begin{equation}
    \label{eq:f_k limit}
     \limsup_{N\to\infty}\left| \dfrac{1}{N}\sum_{n=1}^N f (T^{pn}x) f(T^{qn}x) \right| < \varepsilon,
  \end{equation}
  then, for almost all $x \in X$, we have
  \begin{equation}
    \label{eq:almost orthogonality for f_k}
     \limsup_{N\to\infty} \left|\dfrac{1}{N} \sum_{n=1}^N \bnu(n) f(T^{n}x) \right| < 2\sqrt{\varepsilon\log1/\varepsilon}.
  \end{equation}
\end{Prop} 
\TODOH{Let us notice that we state Katai-Bourgain-Sarnak-Ziegler criterion in the form that we will use. This form can be derived from 
the original one by putting $u(n)= f(T^{pn}x)$ for $x \in X'$ with $\mu(X')=1$. Moreover, a slight modification of the proof 
yields that the Wiener-Wintner version of it holds. But, we do not need such generalization.}

}

Given an arithmetical function $A$ ($A : \N\longrightarrow \C),$ and a positive integer $N \in \N$, for $n \in \{1,\cdots,N\}$,
we define a self-correlation coefficient $c_{n,N}$ of A by

$$ c_{n,N}(A)=\frac1{N}\sum_{m=1}^{N}A(m)A(m+n).
$$
According to Wiener \cite{Wiener}, if the limit exists for each $n$, then this defines the so-called spectral measure of $A$.

\subsection*{Cubic averages and related topics.}

Let $(X,\cb,\Pro)$ be a Lebesgue probability space and given three measure preserving transformations $T_1,T_2,T_3$ on $X$.
Let $f_1,f_2,f_3 \in L^{\infty}(X)$.  The cubic nonconventional ergodic averages of order $2$ with weight $A$ are defined by   $$\frac1{N^2}\sum_{n,m=1}A(n)A(m)A(n+m)f_1(T_1^nx) f_2(T_2^nx) f_3(T_3^nx).$$
This nonconventional ergodic average can be seen as a classical one as follows
$$\frac1{N^2}\sum_{n,m=1}\widetilde{f_1}({\widetilde{T_1}}^n(A,x))
{\widetilde{f_2}}({\widetilde{T_2}}^m(A,x)) {\widetilde{f_3}}({\widetilde{T_3}}^{n+m}(A,x)) ,$$
where $\widetilde{f_i}=\pi_0\otimes f_i, \widetilde{T_i}=(S \otimes T_i),~~~i=1,2,3$ and
$\pi_0$ is define by $x=(x_n)\longmapsto x_0$ on the space $Y=\C^{N}$ equipped with some probability measure.\\

The study of the cubic averages is closely and strongly related to the notion of seminorms introduced in \cite{Gowers} and \cite{Host-K2}. They are nowadays called Gowers-Host-Kra's seminorms.\\

Assume that $T$ is an ergodic measure preserving transformation on $X$. Then,  for any $k \geq 1$,
the Gowers-Host-Kra's seminorms on $L^{\infty}(X)$ are defined inductively as follows
 $$\||f|\|_1=\Big|\int f d\mu\Big|;$$
 $$\||f|\|_{k+1}^{2^{k+1}}=\lim\frac{1}{H}\sum_{l=1}^{H}\||\overline{f}.f\circ T^l|\|_{k}^{2^{k}}.$$

 For each $k\geq 1$, the seminorm $\||.|\|_{k}$ is well defined. For details, we refer the reader to \cite{Host-K2} and \cite{Host-Studia}. Notice that
 the definitions of Gowers-Host-Kra's seminorms can be easily extended to non-ergodic maps as it was mentioned by Chu and
  Frantzikinakis in \cite{chu-nikos}.\\

The importance of the Gowers-Host-Kra's seminorms in the study of the nonconventional multiple ergodic averages
is due to the existence of a $T$-invariant sub-$\sigma$-algebra $\mathcal{Z}_{k-1}$ of $X$ that satisfies
$$\E(f|\mathcal{Z}_{k-1})=0 \Longleftrightarrow \||f|\|_{k}=0.$$
This was proved by Host and Kra in \cite{Host-K2}. The existence of the factors $\mathcal{Z}_{k}$ was also showed by Ziegler in \cite{Ziegler}.
We further notice that Host and Kra established a connection between the $\mathcal{Z}_{k}$ factors and the nilsystems in \cite{Host-K2}.\\

\subsection*{Nilsystems and nilsequences.}  The nilsystems are defined in the setting of homogeneous space \footnote{For a nice account of the theory of the homogeneous space we refer the reader to \cite{Dani},\cite[pp.815-919]{B-katok}.}. Let $G$ be a Lie
group, and $\Gamma$ a discrete cocompact subgroup (Lattice, uniform subgroup) of $G$. The homogeneous space is given by
$X=G/\Gamma$ equipped with the Haar measure $h_X$ and the canonical complete $\sigma$-algebra
$\cb_{c}$. The action of $G$ on $X$ is by the left translation, that is, for any $g \in G$,
we have $T_g(x\Gamma)=g.x\Gamma=(gx)\Gamma.$ If further $G$ is a nilpotent Lie group of order
$k$, 
$X$ is said to be a $k$-step nilmanifold. For any fixed $g \in G$, the dynamical system $(X,\cb_{c},h_X,T_g)$
is called a $k$-step nilsystem. The basic $k$-step  nilsequences on $X$ are defined by $f(g^nx\Gamma)=(f \circ T_g^n)(x\Gamma)$,
where $f$ is a continuous function of $X$. Thus, $(f(g^nx\Gamma))_{n \in \Z}$ is any element of
$\ell^{\infty}(\Z)$, the space of bounded sequences, equipped with uniform norm
$\ds \|(a_n)\|_{\infty}=\sup_{n \in \Z}|a_n|$. A $k$-step nilsequence, is a uniform limit of basic $k$-step nilsequences.
For more details on the nilsequences we refer the reader to  \cite{Host-K3}
and \cite{BHK}\footnote{The term 'nilsequence' was coined by Bergleson-Host and Kra in 2005 \cite{BHK}.}.\\

Recall that the sequence of subgroups $(G_n)$ of $G$ is a filtration if  $G_1=G,$ $G_{n+1}\subset G_{n},$ and
$[G_n,G_p] \subset G_{n+p},$ where $[G_n,G_p]$ denotes the subgroup of $G$ generated by the commutators $[x,y]=
x~y~x^{-1}y^{-1}$ with $x \in G_n$ and $y \in G_p$.
The lower central filtration is given by $G_1=G$ and $G_{n+1}=[G,G_n]$. It is well know that the lower central filtration allows to construct a Lie algebra $\textrm{gr}(G)$ over the ring $\Z$ of integers. $\textrm{gr}(G)$ is called a
graded Lie algebra associated to $G$ \cite[p.38]{Bourbaki2}. The filtration is said to be of degree or length $l$ if
$G_{l+1}=\{e\},$ where $e$ is the identity of $G$.
We denote by $G^e$ the identity component of $G$. Since $X=G/\Gamma$ is compact, we can assume that $G/G^e$ is finitely generated \cite{Leib}.\\

If $G$ is connected and simply-connected  with Lie algebra $\mathfrak{g}$ \footnote{By Lie's fundamental theorems and up
to isomorphism, $\mathfrak{g}=T_eG$, where $T_eG$ is the tangent space at the identity $e$ \cite[p.34]{Kirillov}.},
then $\exp~~:~~G \longrightarrow \mathfrak{g}$ is a diffeomorphism, where $\exp$ denotes the Lie group exponential map.
We further have, by Mal'cev's criterion, that $\mathfrak{g}$ admits a basis $\mathcal{X}=\{X_1,\cdots,X_m\}$ with rational structure constants \cite{Malcev}, that is,
$$[X_i,X_j]=\sum_{n=1}^{m} c_{ijn}X_n,~~~~~~\textrm{for~~all~~~}  1 \leq i,j \leq k,$$
where the constants $c_{ijn}$ are all rational. \\

Let $\mathcal{X}=\{X_1,\cdots,X_m\}$ be a Mal'cev basis of  $\mathfrak{g}$, then any element $g \in G$ can be uniquely written in the form $g=\exp\Big(t_1X_1+t_2X_2+\cdots+t_m X_m\Big),$ $t_i \in \R$, since the map $\exp$ is a diffeomorphism.
The numbers $(t_1,t_2,\cdots,t_k)$ are called the Mal'cev coordinates of the first kind of $g$. In the same manner,
$g$  can be uniquely written in the form $g=\exp(s_1X_1). \exp(s_2X_2).\cdots.\exp(s_m X_m),$ $s_i \in \R$.
The numbers $(s_1,s_2,\cdots,s_k)$ are called the Mal'cev coordinates of the second kind of $g$. Applying Baker-Campbell-Hausdorff formula, it can be shown that the multiplication law in $G$ can be expressed by a polynomial mapping
$\R^m \times \R^m \longrightarrow \R^m$ \cite[p.55]{Onishchik}, \cite{Green-Tao-Orbit}. This gives that any polynomial sequence $g$ in $G$ can be written as follows
$$g(n)=\gamma_1^{p_1(n)},\cdots,\gamma_m^{p_m(n)},$$
where $\gamma_1 ,\cdots,\gamma_m \in G$, $p_i ~~:~~\N \longrightarrow \N$ are polynomials \cite{Green-Tao-Orbit}. Given $n, h \in \Z$, we put
$$\partial_{h}g(n)=g(n+h)g(n)^{-1}.$$
This can be interpreted as a discrete derivative on $G$. Given a filtration $(G_n)$ on $G$,
 a sequence of polynomial $g(n)$ is said to be adapted to $(G_n)$ if  $\partial_{h_i}\cdots \partial{h_1}g$ takes values in $G_i$ for all positive integers $i$ and for all choices of $h_1, \cdots,h_i \in \Z$. The set of all polynomial sequences adapted to $(G_n)$ is denoted by ${\textrm{poly}}(\Z,(G_n))$.\\

  Furthermore, given a Mal'cev's basis ${\mathcal{X}}$ one can induce a right-invariant metric $d_{\mathcal{X}}$ on $X$ \cite{Green-Tao-Orbit}. We remind that for a real-valued function $\phi$ on $X$, the Lipschitz norm is defined by
$$\|\phi\|_{L}=\|\phi\|_{\infty}+\sup_{x \neq y}\frac{\big|\phi(x)-\phi(y)\big|}{d_{\mathcal{X}}(x,y)}.$$
The set $\mathcal{L}(X,d_{\mathcal{X}})$ of all Lipschitz functions is a normed vector space, and for any  $\phi$ and $\psi$ in $\mathcal{L}(X,d_{\mathcal{X}})$, $\phi \psi \in \mathcal{L}(X,d_{\mathcal{X}})$ and
$\|\phi \psi \|_L \leq \|\phi\|_L \|\psi\|_L$. We thus get, by Stone-Weierstrass theorem, that the subsalgebra $\mathcal{L}(X,d_{\mathcal{X}})$ is dense in the space of continuous functions $C(X)$ equipped with uniform norm $\|\|_{\infty}$. 
It turns out that for a Lipschitz function, the extension from an arbitrary subset is possible without
increasing the Lipschitz norm, thanks to Kirszbraun-Mcshane extension theorem \cite[p.146]{Dudley}.\\

In this setting, we remind the following fundamental Green-Tao's theorem on the strong orthogonality of the M\"{o}bius
function to any $m$-step nilsequence, $m \ge 1$.
\begin{Prop}\label{Green-Tao-th}\cite[Theorem 1.1]{Green-Tao}.
Let $G/\Gamma$ be a $m$-step  nilmanifold for some $m\ge 1$.
Let $(G_p)$ be a filtration of $G$ of degree $l \ge 1$. Suppose that $G/\Gamma$ has a $Q$-rational Mal'cev basis $\mathcal{X}$ for some $Q \ge 2$, defining
a metric $d_{\mathcal{X}}$ on $G/\Gamma$. Suppose that $F : G/\Gamma\rightarrow [-1, 1]$ is a Lipschitz function. Then, for any $A>0$, we have the bound,
$$\underset{g \in {\textrm{poly}}(\Z, (G_{p}))}{\Sup}\Big|\frac1{N}\sum_{n=1}^{N}\bmu(n)F(g(n)\Gamma)\Big| \leq C\frac{(1 + ||F||_{L})} {\log^{A} N},$$
where the constant $C$ depends on $m,l,A,Q$, $N \geq 2$.
 \end{Prop}
We further need the following decomposition theorem due to Chu-Frantzikinakis and Host from \cite[Proposition 3.1]{chu-Nikos-H}.  
\begin{Prop}[NSZE-decomposition theorem \cite{chu-Nikos-H}]\label{NSZE} Let $(X,\ca,\mu,T)$ be a dynamical system, $f \in  L^{\infty}(X)$,
and $k \in \N$. Then for every $\varepsilon > 0$, there exist measurable functions $f_{ns} ,f_z ,f_e $, such that
\begin{enumerate}[(a)]
\item $\|f_{\kappa}\|_{\infty} \leq 2 \|f\|_{\infty}$ with $\kappa \in \{ns,z,e\}$.
\item $f = f_{ns} + f_z + f_e$ with  $|\|f_z\|_{k+1} = 0;~~ \|f_e\|_1 <\varepsilon;$ and
\item for $\mu$ almost every $x \in X$, the sequence $(f_{ns}(T^nx))_{n \in \N}$ is a $k$-step nilsequence.
\end{enumerate}
\end{Prop}

\section{Main results and the proof of Theorem \ref{main}.}\label{sec:ms}

\subsection{The first main result and its proof.}
We state our first main result as follows.
\begin{Th}\label{main}
The cubic nonconventional ergodic averages of any order with a bounded multiplicative function
weight converge almost surely to zero provided that the multiplicative function satisfies a strong
Daboussi-Delange condition, that is, for any $k \geq 1$, for any  $f_i \in L^{\infty}(X)$,
$i=1,\cdots, k$, for almost all $x$, we have
\begin{equation}\label{con}
\dfrac1{N^k}\sum_{{\boldsymbol{n}} \in [1,N]^k}\prod_{{\boldsymbol{e}}
\in C^*}\bnu({\boldsymbol{n}}.{\boldsymbol{e}})f_{{\boldsymbol{e}}}
\big(T_{{\boldsymbol{e}}}^{{\boldsymbol{n}}.{\boldsymbol{e}}}x\big) \tend{N}{+\infty}0,
\end{equation}
where ${\boldsymbol{n}}=(n_1,\cdots,n_k)$, ${\boldsymbol{e}}
=(e_1,\cdots,e_k)$, $C^*=\{0,1\}^k\setminus\{(0,\cdots,0)\}$,
${\boldsymbol{n}}.{\boldsymbol{e}}$ is the usual inner product, and $\bnu$ is the multiplicative function satisfying a strong
Daboussi-Delange condition.
 \end{Th}
The proof of our first main result (Theorem \ref{main}) for $k \geq 2$ is essentially based on the inverse Gowers norms theorem  due to Green,
Tao and Ziegler \cite{Green-Tao-Z} combined with the recent result of
Zorin-Kranich on the extension of Wiener-Wintner's version of Bourgain double recurrence
theorem to the nilsequence case \cite{Zoric} (WWBDRT for short)
. Precisely, we will need the following results.
\begin{Prop}[Discrete inverse theorem for Gowers norms \cite{Green-Tao-Z}]\label{DITGN} Let $N \geq 1$
and $s \geq 1$ be integers, and let $\delta>0$. Suppose
$f~:~ \Z \longrightarrow [-1,1]$ is a function supported on $\{1,\cdots,N\}$ such that
$$\dfrac1{N^{s+2}}\sum_{{(n,\boldsymbol{n})} \in [1,N]^{s+1}}\prod_{{\boldsymbol{e}} \in \{0,1\}^{s+1}}f\big(n+{\boldsymbol{n}}.{\boldsymbol{e}}\big) \geq \delta.$$
Then there exists a filtered nilmanifold $G/\Gamma$ of degree $ \leq s$ and complexity $O_{s,\delta}(1)$, a
polynomial sequence $g~~:~~\Z \longrightarrow G$, and a Lipschitz function $F~~:~~G/\Gamma \longrightarrow \R$ of
Lipschitz constant $O_{s,\delta}(1)$ such that
$$\frac1{N}\sum_{n=1}^{N}f(n) F(g(n)\Gamma) \gg_{s, \delta} 1.$$
\end{Prop}
For the definition of complexity, we refer to \cite{Green-Tao-Z}. The second result
that we will need is the following 
\begin{Prop}[Nilsequence's version of WWBDRT \cite{Zoric}]\label{Pavel} Let $(X,\mu,T)$ be an ergodic
dynamical system and $a,b$ distinct non-zero integers. Then,
 for any $f,g \in L^{\infty}$,  there exists a measurable set $X'$ of full measure such
 that for any filtration $(G_{p})$ of length $l \geq 2$ and for any function $F$ in the
 Sobolev space $W^{r,2^l}$,for any $x \in X'$,  we have
$$ \underset{g \in {\textrm{poly}}(\Z, (G_{p}))}{\sup}\Big|{\frac1{N}}\sum_{n=0}^{N-1}f_1(T^{an}x)f_2(T^{bn}x) F(g(n)\Gamma)\Big|
\leq C.\min\{\|f_1\|_{U^{2+l}}^{2^{l+1}},\|f_2\|_{U^{2+l}}^{2^{l+1}}\},$$
where $r=\sum_{m=1}^{l}(d_m-d_{m+1})\binom{l}{m-1}$ with $d_i$ is the dimension of $G_i$
and $C$ is a constant which depends only on the nilmanifold $G/\Gamma.$
\end{Prop}

At this point, let us give a proof of our main result.\\

\noindent {\textbf{Proof of Theorem \ref{main}.}} 
Notice that the strong Daboussi-Delange condition
implies that $\bnu$ is aperiodic, that is,
$$\frac{1}{N}\sum_{n=1}^{N}\bnu(a.n+b) \tend{N}{+\infty}0,~~~~~~~~~~\textrm{for~~~all}~~a,b\in \N.$$
Therefore, by Theorem 2.2 from \cite{Nikos-Host3}, for any nilsequence $(u_n)$, we have
$$\frac{1}{N}\sum_{n=1}^{N}\bnu(n)u_n \tend{N}{+\infty}0.$$

Now, for the case $k=1$, we refer to \cite[Section 3]{al-lem}. Let us assume from now that $k \geq 2$.\\

We proceed by contradiction. Assume that  (\ref{con}) does not hold. Then, there exist $\delta>0$, a functions $f_e,$ $e \in V_k$ and $\mu(A)>0$
such that for each $x\in A$
 \[
\limsup_{N \longrightarrow +\infty}\dfrac1{N^k}\sum_{{\boldsymbol{n}} \in [1,N]^k}\prod_{{\boldsymbol{e}} \in C^*}\bnu({\boldsymbol{n}}.{\boldsymbol{e}})f_{{\boldsymbol{e}}}
\big(T_{{\boldsymbol{e}}}^{{\boldsymbol{n}}.{\boldsymbol{e}}}x\big) \geq \delta,
\]
where $\nu$ is a bounded multiplicative function which satisfies
a strong Daboussi-Delange condition. Whence, by \cite[Proposition 3.2]{chu-nikos}, we have
$$\||\bnu(n) f_e\circ T_e^n(x)|\|_{U^{k+1}} \geq \delta,$$
for some $e \in V_k$ and  for any $x \in A$.\\

This combined with Proposition \ref{DITGN} yields that there exist a
nilmanifold $G/\Gamma$, a filtration $(G_p)$, a
polynomial sequence $g$ and a Lipschitz function $F$ such that
$$\frac1{N}\sum_{n=1}^{N}\bnu(n) f_e(T_e^{n}x) F(g(n)\Gamma) \gg_{k, \delta} 1.$$
Applying the decomposition theorem (Proposition \ref{NSZE}), we
way write $f_e=f_{e,ns}+f_{e,z}+f_{e,e}$ and may assume that we have
$|\|f_{e,z}(T_e^n(x))\||_{U^{k+1}}=0$. Notice that we further have

$$\frac1{N}\sum_{n=1}^{N}\bnu(n) f_{e,ns}(T_e^{n}x) F(g(n)\Gamma) \tend{N}{+\infty}0,$$
since the pointwise product of two nilsequences is a nilsequence and $\bnu$ satisfies the strong Daboussi-Delange condition.
\\

We thus get, up to some small error, that
 $$\frac1{N}\sum_{n=1}^{N}\bnu(n) f_e(T_e^{n}x) F(g(n)\Gamma)\sim
 \frac1{N}\sum_{n=1}^{N}\bnu(n) f_{e,z}(T_e^{n}x) F(g(n)\Gamma).$$

Applying Katai-Bourgain-Sarnak-Ziegler criterion,
we get that there exist distinct prime numbers $p,q$ such that
$$\limsup \Big|\frac1{N}\sum_{n=1}^{N}f_{e,z}\circ T_{e}^{pn}(x) f_{e,z}\circ T_e^{qn}(x)  F(g(np)\Gamma)  F(g(nq)\Gamma)\Big| \gg_{k, \delta} 1$$
for a subset $B$ with positive measure.
This contradicts Proposition \ref{Pavel} since $\mu(B)>0$ and $||f_{e,z}(T_e^n(x))||_{U^{k+1}}=0$, and the proof of the theorem is complete.
\QEDA

\subsection{Other main results}
Our second main result can be stated as follows
\begin{Th}\label{main-2}
Assume that for some $i \in \{1,2,3\}$ (resp. $i \in \{1,2,3,4,5,6,7\}$), $T_i$ has a discrete spectrum.
Then, the cubic  ergodic averages of order $2$ (resp. of order $3$) with
multiplicative function weight which satisfies the strong Daboussi-Delange condition converge almost surely to $0$.

\end{Th}

Obviously, we have the following corollary.


\begin{Cor}\label{main-C2}
Assume that for some $i \in \{1,2,3\} $ (resp. $i \in \{1,2,3,4,5,6,7\}$), $T_i$ is a nilsystem. Then,
the cubic ergodic averages of order $2$ (resp. of order 3) with the M\"{o}bius or Liouville function
weight converge almost surely to $0$.
\end{Cor}


The proof of Corollary \ref{main-C2} can be obtained independently and we shall give
such proof in section 4.





The proof of Theorem \ref{main-2} is divided into subsections in section \ref{sec-4}. In subsection \ref{sub4-1}, we prove that for any $N \geq 2$,
$$\frac1{N}\sum_{m=1}^{N}\Big|\frac1{N}\sum_{n=1}^{N} \bnu(n) \bnu(n+m)\Big| \leq \frac{C}{\log(N)^{\epsilon}},$$
and
$$\frac1{N^2}\sum_{n,p=1}^{N}\Big|\frac1{N}\sum_{m=1}^{N} \bnu(m) \bnu(n+m)\bnu(m+p)\bnu(n+m+p)\Big| \leq \frac{C}{\log(N)^{\varepsilon}},$$
where $C,\varepsilon$ are some positive constants and $\bnu$ is 
a bounded multiplicative function which satisfies a weak Daboussi-Delange condition, that is,
$$\Big|\frac1{N}\sum_{n=1}^{N}\bnu(n)e^{2i\pi n \alpha}\Big| \leq \frac{c}{\log(N)^{\epsilon}},~~~~\textrm{for~~all~~} \alpha,$$
with $c$ is a positive constant.

In subsection \ref{sub4.4}, we provide a simpler proof of Corollary \ref{main-C2} based on the proof of our second main result.


Notice that our result involves the self-correlation of $\bmu$. Furthermore, as far as we know the problem of
the self-correlation of $\bmu$ is still open. This problem is known as Elliott's conjecture and it can be stated as follows.
\begin{conj}[of Elliott]\cite{Elliott-C}
 \begin{eqnarray}\label{Elliot}
  \lim_{N \longrightarrow +\infty} c_{h,N}(\bmu)=\begin{cases}
0 &{\rm {~~if~~}} h \neq 0 \\
\displaystyle \frac{6}{\pi^ 2} &{\rm {~~if~not~}}.
\end{cases}
 \end{eqnarray}
\end{conj}
P.D.T.A. Elliott wrote in his 1994's AMS Memoirs that ``even the simple particular
cases of the correlation (when $h=1$ in \eqref{Elliot}) are not well understood.
Almost surely the M\"obius function satisfies \eqref{Elliot} in this case, but at
the moment  we are unable to prove it.''\\

Recently, el Abdalaoui and Disertori in \cite{al-Diser} established that the $L^1$-flatness of
trigonometric polynomials with M\"{o}bius coefficients implies Elliott's conjecture. The $L^1$-flatness
of trigonometric polynomials with coefficients in $\{0,1,-1\}$ is an open problem in harmonic analysis
and spectral theory of dynamical systems (see \cite{al-Nad} and the references therein).\\

Nevertheless, as a consequence of Theorem \ref{main}, we have the following result
\begin{Cor} Let $\bnu$ be the M\"{o}bius function or the Liouville function. Then
\begin{equation}\label{MLcubic}
\dfrac1{N^k}\sum_{{\boldsymbol{n}} \in [1,N]^k}\prod_{{\boldsymbol{e}}
\in C^*}\bnu({\boldsymbol{n}}.{\boldsymbol{e}})
\tend{N}{+\infty}0,
\end{equation}
where ${\boldsymbol{n}}=(n_1,\cdots,n_k)$, ${\boldsymbol{e}}
=(e_1,\cdots,e_k)$, $C^*=\{0,1\}^k\setminus\{(0,\cdots,0)\}$,
${\boldsymbol{n}}.{\boldsymbol{e}}$ is the usual inner product.
\end{Cor}
\begin{proof}Take $f_{\boldsymbol{e}}=1$ in Theorem \ref{main}.
\end{proof}

Our third main result can be stated as follows.
\begin{Th}\label{Main-Speed1}Let $\bnu$ be a multiplicative function satisfying the Daboussi-Delange condition with logarithmic speed.
Then, for any $N \geq 2$,
$$\frac1{N}\sum_{m=1}^{N}\Big|\frac1{N}\sum_{n=1}^{N} \bnu(n) \bnu(n+m)\Big| \leq \frac{C}{\log(N)^{\kappa}},$$
and
$$\frac1{N^2}\sum_{n,p=1}^{N}\Big|\frac1{N}\sum_{m=1}^{N} \bnu(m) \bnu(n+m)\bnu(m+p)\bnu(n+m+p)\Big| \leq \frac{C}{\log(N)^{\kappa}},$$
where $C$ is some positive constant.
\end{Th}

It is easy to see that from Davenport estimation \eqref{Man} that we have the following
\begin{Cor}\label{Main-Speed2} Let $\Lambda$ be the von Mangoldt function. Then, for any $N \geq 2$, for almost all $t$,
$$\Big|\frac1{N}\sum_{n=1}^{N} \Lambda(n) e^{2 i\pi n  t }\Big| \tend{n}{+\infty}0.$$
\end{Cor}

We further have, by the proof of Theorem \ref{Main-Speed1}, the following corollary,

\begin{Cor}\label{Main-P} For any $\rho > 1$ there exist two sequences of integer $m_l,n_l$ such that
$$\frac{1}{[\rho^{m_l}]}\sum_{k=1}^{[\rho^{m_l}]}\bmu(k)\bmu(k+n_l)\tend{l}{+\infty} 0,$$
and
$$\frac{1}{[\rho^{m_l}]}\sum_{k=1}^{[\rho^{m_l}]}\lamob(k)\lamob(k+n_l)\tend{l}{+\infty} 0.$$ We further have, for any integer
$N \geq 2$, for any $\epsilon >0$,
$$ \frac1{N}\sum_{n=1}^{N}\Big|\frac1{N}\sum_{m=1}^{N}\bmu(m)\bmu(n+m)\Big|\leq\frac{C}{\log(N)^{\epsilon}},$$
and
$$\frac1{N}\sum_{n=1}^{N}\Big|\frac1{N}\sum_{m=1}^{N}\lamob(m)\lamob(n+m)\Big|\leq \frac{C}{\log(N)^{\epsilon}},$$
where $C$ is a constant which depends only on $\varepsilon<1.$
\end{Cor}
This gives
$$\frac1{N}\sum_{n=1}^{N}\Big|\frac1{N}\sum_{m=1}^{N}\bmu(m)\bmu(n+m)\Big|\tend{N}{+\infty}0,$$
and
$$\frac1{N}\sum_{n=1}^{N}\Big|\frac1{N}\sum_{m=1}^{N}\lamob(m)\lamob(n+m)\Big|\tend{N}{+\infty}0.$$
Before starting to prove our main results, let us point out that it suffices to establish our results for
a dense set of functions. Indeed, assume that the convergence holds for some $f_{\overline{1}}$,
$\overline{1}=(1,0,\cdots,0)$, and let $g$ be such that $||f_{\overline{1}}-g||_1<\epsilon$, for a given $\epsilon>0.$ Put
$$\psi_N(f_{\overline{1}})=\frac{1}{N}
\sum_{n_1=1}^{N}
\bnu(n_1)f_{\overline{1}}(T_{\overline{1}}^{n_1}(x))\frac{1}{N^{k-1}}\sum_{\overset{\boldsymbol{m}\in [1,N]^{k-1}}{\boldsymbol{n}=(n_1,\boldsymbol{m})}}\prod_{e \in C^*\setminus\{\overline{1}\}}
\bnu(\boldsymbol{n}.e)f_e(T_e^{\boldsymbol{n}.e}(x)).
$$
Then,
\begin{eqnarray*}
&&|\psi_N(f_{\overline{1}})-\psi_N(g)|\\
 &\leq& \frac1{N}\sum_{n_1=1}^{N}|f_{\overline{1}}(T_{\overline{1}}^{n_1}x)-g(T_{\overline{1}}^{n_1}x)|
\Big|\frac{1}{N^{k-1}}\sum_{\overset{\boldsymbol{m}\in [1,N]^{k-1}}{\boldsymbol{n}=(n_1,\boldsymbol{m})}}\prod_{e \in C^*\setminus\{\overline{1}\}}
\bnu(\boldsymbol{n}.e)f_e(T_e^{\boldsymbol{n}.e}(x))\Big|\\
&\leq& c \Big(\prod_{e \in C^*\setminus\{\overline{1}\}}||f_e||_{\infty}\Big)\frac1{N}\sum_{n=1}^{N}|f_1(T_1^nx)-g(T_1^nx)|.
\end{eqnarray*}
Letting $N \longrightarrow \infty$, it follows
$$\limsup|\psi_N(f_{\overline{1}})-\psi_N(g)| \leq \|f_{\overline{1}}-g\|_1 \leq \epsilon,$$
by Birkhoff ergodic theorem combined with our assumption. Notice that the maps $T_e$ are
supposed to be ergodic. From now on, without loss of generality, we will assume that $T_e$ are ergodic.

\section{The proof of Theorems \ref{main-2} and \ref{Main-Speed1}}\label{sec-4}
In this subsection we will show Theorems \ref{main-2} and \ref{Main-Speed1}. In the process to do
we also provide the proof of Corollary \ref{Main-P}.

\subsection{The case $k=2$.}\label{sub4-1}  
In this subsection, we focus our study on the case $k=2$ and at least one of the maps $T_i, i=1,2,3$ has a
discrete spectrum. 

\begin{proof}[{\textbf {Proof of Theorem \ref{main-2}}}]Let assume that $T_1$ has a discrete
spectrum and let $f_1$ be an eigenfunction with eigenvalue $\lambda$. Then, for almost all  $x \in X$, we can write
\begin{eqnarray}\label{bo1}
&&\Big|\frac1{N^2}\sum_{n,m=1}^{N}\bnu(n)\bnu(m)\bnu(n+m)
f_1(T_1^nx) f_2(T_2^mx)f_3(T_3^{n+m}x)\Big|\nonumber \\
&=& \Big|\frac1{N}\sum_{m=1}^{N}\bnu(m)f_2(T_2^mx)
\frac1{N} \sum_{n=1}^{N}\bnu(n) \bnu(n+m)\lambda^n
 f_3(T_3^{n+m}x)\Big| \nonumber\\
&=& \Big|\frac1{N}\sum_{m=1}^{N}\bnu(m)f_2(T_2^mx)
\frac1{N} \sum_{n=1}^{N}\bnu(n) \bnu(n+m)\lambda^{n+m}
 f_3(T_3^{n+m}x)\Big| \nonumber\\
 & \leq &  \frac1{N} \sum_{m=1}^{N} \big|f_2(T_2^mx)\big|\Big|
 \frac1{N} \sum_{n=1}^{N}\bnu(n)\bnu(n+m) \lambda^{n+m}
f_3(T_3^{n+m}x)\Big|\nonumber\\
 &\leq& \|f_2\|_{\infty} \frac1{N} \sum_{m=1}^{N} \Big|\frac1{N}\sum_{n=1}^{N} \bnu(n)\bnu(n+m) \lambda^{n+m}
f_3(T_3^{n+m}x)\Big|
\end{eqnarray}

Applying Cauchy-Schwarz inequality we can rewrite \eqref{bo1} as
\begin{eqnarray}\label{boC1}
&&\Big|\frac1{N^2}\sum_{n,m=1}^{N}\bnu(n)\bnu(m)\bnu(n+m)
f_1(T_1^nx) f_2(T_2^mx)f_3(T_3^{n+m}x)\Big|\nonumber \\
&\leq &\|f_2\|_{\infty} \Big(\frac1{N} \sum_{m=1}^{N} \Big|\frac1{N}\sum_{n=1}^{N} \bnu(n)\bnu(n+m) \lambda^{n+m}
f_3(T_3^{n+m}x)\Big|^2\Big)^{\frac1{2}}
\end{eqnarray}

Furthermore, by Bourgain's observation \cite[equations (2.5) and (2.7)]{BourgainD}, we have
\begin{eqnarray}\label{bo2}
&&\sum_{m=1}^{N}\Big|\frac1{N}\sum_{n=1}^{N} \bnu(n)\bnu(n+m) \lambda^{n+m}f_3(T_1^{n+m}x)\Big|^2 \nonumber\\
 &=& \sum_{m=1}^{N}\Big|\int_{\T}\Big(\frac1{N}\sum_{n=1}^{N}
 \bnu(n)z^{-n}\Big)\Big(\sum_{p=1}^{2N} \bnu(p)
 f_3(T_3^px) {\big(\lambda.z\big)}^p\Big) z^{-m} dz\Big|^2 \nonumber \\
 &\leq&  \int_{\T}\Big|\frac1{N}\sum_{n=1}^{N} \bnu(n)
 z^{-n}\Big|^2\Big|\sum_{p=1}^{2N} \bnu(p)
 f_3(T_3^px) {\big(\lambda.z\big)}^p\Big|^2 dz.
\end{eqnarray}
The last inequality is due to Parseval-Bessel inequality. Indeed, put
$$\Phi_N(z)=\Big(\frac1{N}\sum_{n=1}^{N}
 \bnu(n) z^{-n}\Big)\Big(\sum_{p=1}^{2N} \bnu(p)
 f_3(T_3^px) {\big(\lambda.z\big)}^p \Big).$$
Then, for any $m \in \Z$,
 $$ \widehat{\Phi_N}(m)=\int_{\T}\Big(\frac1{N}\sum_{n=1}^{N}
 \bnu(n)z^{-n}\Big)\Big(\sum_{p=1}^{2N} \bnu(p)
 f_3(T_3^px) {\big(\lambda.z\big)}^p\Big) z^{-m} dz.$$
Whence
\begin{eqnarray*}
 &&\sum_{m=1}^{N}\Big|\int_{\T}\Big(\frac1{N}\sum_{n=1}^{N}
 \bnu(n) z^{-n}\Big)\Big(\sum_{p=1}^{2N} \bnu(p)
 f_3(T_3^px) {\big(\lambda.z\big)}^p\Big) z^{-m} dz\Big|^2\\
  &=&\sum_{m=1}^{N}\Big|\widehat{\Phi_N}(m)\Big|^2\\
 &\leq& \int_{\T} |\Phi_N(z)|^2 dz.
 \end{eqnarray*}
Now, combining \eqref{boC1} with \eqref{bo2} we can assert that
\begin{eqnarray}\label{bo3}
&&\Big|\frac1{N^2}\sum_{n,m=1}^{N}\bnu(n)\bnu(m)\bnu(n+m)
f_1(T_1^nx) f_2(T_2^mx)f_3(T_3^{n+m}x)\Big|\nonumber\\&=& \Big|\frac1{N}\sum_{m=1}^{N}\bnu(m)f_2(T_2^mx)
\frac1{N} \sum_{n=1}^{N}\bnu(n) \lambda^n\bnu(n+m)
 f_3(T_3^{n+m}x)\Big|\nonumber\\
 &\leq&\|f_2\|_{\infty} \Big(\frac1{N} \sup_{z \in \T} \Big|\frac1{N}\sum_{n=1}^{N} \bnu(n)
 z^{-n}\Big|^2 \int_{\T}\Big|\sum_{p=1}^{2N} \bnu(p)
 f_3(T_3^px) {(\lambda.z)}^p\Big|^2 dz\Big)^{\frac12},
\end{eqnarray}
We thus get
\begin{eqnarray}\label{bo4}
&&\Big|\frac1{N^2}\sum_{n,m=1}^{N}\bnu(n)\bnu(m)\bnu(n+m)
f_1(T_1^nx) f_2(T_2^mx)f_3(T_3^{n+m}x)\Big|\nonumber\\
 &\leq&\|f_2\|_{\infty}  \sup_{z \in \T} \Big|\frac1{N}\sum_{n=1}^{N}\bnu(n)
z^{-n}\Big| . \Big(\frac1{N} \sum_{p=1}^{2N}|\bnu(p)|^2\Big)^{\frac12} \|f_3\|_{\infty},
\end{eqnarray}
since the map $z \mapsto \lambda z$ is a measure-preserving transformation, and we have
$$\ds \int_{\T}\Big|\sum_{p=1}^{2N} \bnu(p) f_3(T_3^px) z^p\Big|^2 dz=\sum_{p=1}^{2N}|\bnu(p)|^2|f_3(T_3^px)|^2
\leq \sum_{p=1}^{2N}|\bnu(p)|^2  \|f_3\|_{\infty}^2. $$

It follows from our assumption on the multiplicative function $\bnu$ that we have
\begin{eqnarray}\label{eq:Dav3}
&&\Big|\frac1{N^2}\sum_{n,m=1}^{N}\bnu(n)\bnu(m)\bnu(n+m)
f_1(T_1^nx) f_2(T_2^mx)f_3(T_3^{n+m}x)\Big|\nonumber\\
&\leq& \sqrt{2}\|f_2\|_{\infty}\|f_3\|_{\infty}\frac{C}{\log(N)^{\varepsilon}}.
\end{eqnarray}
where $C$ and $\varepsilon$ are positive constants, and $\varepsilon<1$. Letting $N \longrightarrow \infty$, we conclude that the
almost sure convergence holds.
The proof of the theorem for the case $k=2$ is complete.
\end{proof}

\subsection{On the self-correlation of M\"{o}bius and Louiville functions.}\label{sub4.2}    
Before the proof of Theorem \ref{main-2} for $k=3$,  we discuss the self-correlation of M\"{o}bius and Louiville functions first.
Notice that we have actually proved
\begin{factor} Let $\bnu$ be the M\"{o}bius or the Liouville function. Then, for any $N \geq 2$,
\begin{eqnarray}\label{Dav3}
\frac1{N}\sum_{n=1}^{N}\Big|\frac1{N}\sum_{m=1}^{N} \bnu(n)\bnu(n+m)\Big| \leq \frac{C}{\log(N)^{\varepsilon}},
\end{eqnarray}
where $C$ and $\varepsilon$ are  positive constants, $\varepsilon<1$.
\end{factor}
\begin{proof}
  This follows by combining \eqref{bo1}, \eqref{bo2}, \eqref{bo3}, \eqref{MS}
with \eqref{Man} and by taking $\lambda = 1$ and $f_3 = 1$.
\end{proof}
At this point, the proof of the second part of Corollary \ref{Main-P} follows.\\

Let $\rho>1$, then for $N=[\rho^m]$ with some $m\ge1$, \eqref{Dav3} takes the form

\[
 \frac1{[\rho^m]}\sum_{n=1}^{[\rho^m]}|c_{n,N}|\\ \leq \frac{C}{{\rho}^{m\varepsilon}}\text{,~~~~~for some positive number~~ }\varepsilon.
\]
This gives
\[
\sum_{ m \geq 1}\frac1{[\rho^m]}\sum_{n\leq [\rho^m]}|c_{n,[\rho^m]}| <+\infty.
\]
Let $(\delta_l)$ be a sequence of positive numbers such that $\delta_l \tend{l}{+\infty}0.$ Then, for any $l \geq 1$
there exists a positive integer $m_l$ such that
\[
\sum_{ m \geq m_l}\frac1{[\rho^m]}\sum_{n\leq [\rho^m]}|c_{n,[\rho^m]}| <\delta_l.
\]
This gives, for any $m \geq m_l$,
\[
\frac1{[\rho^m]}\sum_{n\leq [\rho^m]}|c_{n,[\rho^m]}| <\delta_l.
\]
Hence, there exists $n_l \leq [\rho^m] $ such that
\[
|c_{n_l,[\rho^{m_l}]}| <\delta_l.
\]
By letting $l$ go to $\infty$, we get $c_{n_l,[\rho^{m_l}]} \tend{l}{+\infty} 0$. This prove the first part of Corollary~\ref{Main-P}.\QEDA\\

Note that we have proved more, namely,
\begin{Cor}For any $\rho>1$,
$$\sum_{ m \geq 1}\frac1{[\rho^m]}\sum_{n=1}^{[\rho^m]}\Big|\frac1{[\rho^m]}
\sum_{k=1}^{[\rho^{m}]}\bmu(k)\bmu(k+n)\Big| <+\infty. $$
\end{Cor}


\begin{rem} It is shown in \cite{al-Diser} that if Sarnak's conjecture holds with some technical assumption then the self-corrections of $\bmu$ satisfy
$$\frac{1}{2N}\sum_{n=-N}^{N}\bmu(n)\bmu(n+k) \tend{N}{+\infty} 0,$$
for any $k<0$ \footnote{We extend the definition of $\bmu$ to the negative integer in usual fashion.}.
\end{rem}

\subsection{The case $k=3$.}\label{sub4-3}
In this subsection we give the proof of Theorem \ref{main-2} for $k=3$.
For simplicity, we shall prove the following
\begin{eqnarray*}
\Big|\frac1{N^3} \sum_{n,p,m=1}\bnu(n)\bnu(m)\bnu(p)\bnu(n+m)\bnu(n+p)\bnu(m+p)\bnu(n+m+p)\Big| \leq \frac{C}{\log(N)^{\epsilon}},
\end{eqnarray*}
where $C$ and $\epsilon>0$ are positive constants,  and $\bnu$ is a bounded multiplicative function which satisfies the weak Daboussi-Delange condition.\\

Notice that by the triangle inequality, we have
$$\Big|\frac1{N^3} \sum_{n,p,m=1}\bnu(n)\bnu(m)\bnu(p)\bnu(n+m)\bnu(n+p)\bnu(m+p)\bnu(n+m+p)\Big|$$
$$
\leq \frac1{N^2}\sum_{n,p=1}^{N} \Big|\frac1{N} \sum_{m=1}^{N}\bnu(n)\bnu(m)\bnu(p)\bnu(n+m)\bnu(n+p)\bnu(m+p)\bnu(n+m+p)\Big|.$$
But, since $\bnu$ is bounded (say by one), we can write
$$\frac1{N^2}\sum_{n,p=1}^{N} \Big|\frac1{N} \sum_{m=1}^{N}\bnu(n)\bnu(m)\bnu(p)\bnu(n+m)\bnu(n+p)\bnu(m+p)\bnu(n+m+p)\Big|$$
$$\leq \frac1{N^2}\sum_{n,p=1}^{N} \Big|\frac1{N} \sum_{m=1}^{N}\bnu(m)\bnu(n+m)\bnu(m+p)\bnu(n+m+p)\Big|.$$
Let $p$ be fixed and apply Cauchy-Schwarz inequality combined with Bourgain observation to get
\begin{eqnarray*}
&&\frac1{N}\sum_{n=1}^{N}\Big|\frac1{N} \sum_{m=1}^{N}\bnu(m)\bnu(n+m)\bnu(m+p)\bnu(n+m+p)\Big|\\
&&\leq \Big(\frac1{N}\sum_{n=1}^{N}\Big|\frac1{N} \sum_{m=1}^{N}\bnu(m)\bnu(n+m)\bnu(m+p)\bnu(n+m+p)\Big|^2\Big)^{\frac12}\\
&&\leq \Big(\frac1{N}\sum_{n=1}^{N}\Big|\widehat{\Psi_{N,p}}(n)\Big|^2\Big)^{\frac12}
\end{eqnarray*}
where
$$\Psi_{N,p}(z)=\Big(\frac1{N} \sum_{m=1}^{N}\bnu(m)z^{-m}\Big) \Big(\sum_{m'=1}^{N}\bnu(m')\bnu(m'+p)z^{m'}\Big).$$
As before, applying Bessel-Parseval inequality, we can write
\begin{eqnarray*}
&&\frac1{N}\sum_{n=1}^{N}\Big|\frac1{N} \sum_{m=1}^{N}\bnu(m)\bnu(n+m)\bnu(m+p)\bnu(n+m+p)\Big|\\
&&\leq \Big(\frac1{N}\int  \Big|\Psi_{N,p}(z)\Big|^2 dz\Big)^{\frac12}.\\
\end{eqnarray*}
Now, by our assumption on $\bnu$, we have
$$\Big|\Psi_{N,p}(z)\Big| \leq \frac{C}{\log(N)^{\epsilon}}\Big|\sum_{m'=1}^{N}\bnu(m')\bnu(m'+p)z^{m'}\Big|.$$
Therefore,
$$\frac1{N}\int  \Big|\Psi_{N,p}(z)\Big|^2 dz \leq \frac{C^2}{\log(N)^{2\epsilon}}.\Big(
\frac1{N}\sum_{m'=1}^{N}\bnu^2(m')\bnu^2(m'+p)\Big).$$
It turns out that for any $p$, one can estimate the quantity
$$\frac1{N}\sum_{m'=1}^{N}\bnu^2(m')\bnu^2(m'+p).$$
for $\bnu \in \{\bmu,\lamob\}$, thanks to Mirsky Theorem. But we don't need such estimation. Here, we observe only that $\bnu^2$ is bounded by one.
Whence, for any $p$,
$$\Big(\frac1{N}\int  \Big|\Psi_{N,p}(z)\Big|^2 dz\Big)^{\frac12}
\leq  \frac{C}{\log(N)^{\epsilon}}.$$
This proves  Corollary \ref{Main-P}. A careful application of our previous machinery allows us to prove Theorem \ref{Main-Speed2},
and the detailed verification is left to the reader.\QEDA

\begin{proof}[{\textbf {Proof of Theorem \ref{Main-Speed1}}}] It follows from subsections \ref{sub4-1} and \ref{sub4-3}.
\end{proof}

\subsection{On the nilsystem case.}\label{sub4.4}    
In this subsection, we present an alternative proof of our second main result when $k=2$ or $3$,
$\bnu \in \{\bmu,\lamob\}$ and at least one of the dynamical systems is a nilsystem.\\

\begin{proof}[{\textbf {Proof of Corollary \ref{main-C2}}}]Let us assume that $T_1$ is an elementary nilsystem of order $s$, that is, $T_1$ is an ergodic $s$-step nilsystem on $X=G/\Gamma$, where $G$ is a nilpotent Lie group of dimension $s$ and $\Gamma$ is a discrete subgroup. By the density argument, it suffices to prove the theorem for a nilsequence $(f_1(T^nx))$, $x\in X$, $f_1$ is a continuous function on $X$. Now, by Leibman's observation \cite{Leib}, we can embed $G$ into a connected and simply-connected nilpotent Lie group $\hat{G}$ with a cocompact subgroup $\hat{\Gamma}$ such that $X=G/\Gamma$ is isomorphic to a sub-nilmanifold of $\hat{X}=\hat{G}/\hat{\Gamma},$ with all translations from $G$ represented in $\hat{G}$. Furthermore, by Tietze-Uryshon extension theorem \cite[p.48]{Dudley}, we can extend $f_1$ to $\hat{X}$. Hence, we are reduced to prove our main result for the nilsystems on $\hat{X}$. \\

Analyzing the proof given in subsection \ref{sub4-1}, we need to estimate
$$\sup_{z \in \T} \Big|\frac1{N}\sum_{n=1}^{N} \bmu(n) f_1(T^nx) z^{-n}\Big|.$$
Again, by the density argument, we may assume that $f_1$ is Lipschitz. We further notice that the sequence $(a_nz^{-n})_n$ can viewed as a nilsequence on
$Y=\hat{G}/\hat{\Gamma} \times \R/\T$. But the group $\hat{G}\times \R$ is connected and simply-connected, and the function
$F_1(x,z)=f_1(x)z^{-1}$ is Lipschitz. Then, we can apply Green-Tao's Theorem (Theorem 1.1 in \cite{Green-Tao}) for a given filtration $(H_n)$ of $\hat{G}\times \R$ of length $m \geq 1$. This gives,
$$\sup_{z \in \T} \Big|\frac1{N}\sum_{n=1}^{N} \bmu(n) f_1(T^nx) z^{-n}\Big| \leq C \frac{1+\|f_1\|_{L}}{\log^{A}(N)},$$
For any $A>0$, uniformly on $x$ and $z$. Letting $N$ go to infinity, we get
\begin{eqnarray}\label{important}
\sup_{z \in \T} \Big|\frac1{N}\sum_{n=1}^{N} \bmu(n) f_1(T^nx) z^{-n}\Big| \tend{N}{+\infty}0.
\end{eqnarray}
Whence,
$$\Big|\frac1{N^2}\sum_{n,m=1}^{N}\bmu(n)\bmu(m)\bmu(n+m)
f_1(T_1^nx) f_2(T_2^mx)f_3(T_3^{n+m}x)\Big| \tend{N}{+\infty}0. $$
by \eqref{bo4},  which ends the proof of the theorem for the case $k=2$, the rest of the proof is left to the
reader.
\end{proof}

Note that by (\ref{important}), we have proved the following popular and well-known result.
\begin{Prop}\label{Mobius-Nil} Sarnak's conjecture holds for any nilsystem.
\end{Prop}

\begin{que}A natural problem suggested by our result is the following: Assume that $T$ satisfies Sarnak's conjecture,
 do we have for any continuous function, for all $x \in X$,
\[
\sup_{t}\Big|\frac1{N}\sum_{n=1}^{N}\bmu(n)f(T^nx) e^{2 \pi i n t }\Big| \tend{N}{+\infty}0? \textrm{~~(WWS)~~~}
\]
 Notice that the topological entropy of the cartesian product of two dynamical flow on compact set is the sum of their topological entropy \cite{GoodW}.\\
\end{que}

\subsection{On the cubic average with Mangoldt weight}
The study of the correlations of von Mangoldt function  is of great importance in number theory, since it is related to the
famous old conjecture
of the twin numbers and more generaly to Hard-Littlewood k-tuple conjecture. It is also related to Riemann hypothesis and the Goldbach
conjectures. Here we will establish the following.
\begin{Th}\label{main-Mangoldt}The cubic nonconventional ergodic averages of any order with \linebreak von Mangoldt function
weight converge almost surely provided that the systems are nilsystems.
\end{Th}
The proof of Theorem \ref{main-Mangoldt} is largely inspired from Ford-Green-Konyagin-Tao's proof of the following theorem \cite{FGKtao}. It is used also some
elementary fact on Gowers uniformity semi-norms.
\begin{Prop}[Ford-Green-Konyagin-Tao \cite{FGKtao}]\label{FGKT}The Gowers norm of von Mangoldt function is positive.
Precisely, for any $d \geq 1$,
$$\frac1{N^d}\sum_{\vec{n} \in[1,N]^d}\prod_{e \in C^*}\Lambda(\vec{n}.e)\tend{N}{+\infty}\prod_{p}\beta_p,$$
where
$$\beta_p=\frac1{p^d}\sum_{\vec{n} \in (\Z/p\Z)^d}\prod_{e \in C^*}\Lambda_{\Z/p\Z}(\vec{n}.e).$$
The function $\Lambda_{\Z/p\Z}$ is the local von Mangoldt function, that is, the $p$-periodic function defined setting
$\Lambda_{\Z/p\Z}(b)=\frac{p}{p-1}$ when $b$ is coprime to $p$ and $\Lambda_{\Z/p\Z}(b)=0$ otherwise.
\end{Prop}
As it is mentioned in \cite{FGKtao}, the fundamental ingredients in the proof of Proposition \ref{FGKT} are the M\"{o}bius orthogonality to the
nilsequences (Proposition \ref{Mobius-Nil}) combined with
the inverse Gowers theorem (Proposition \ref{DITGN}) and the``W-trick''. 
We will need also the following lemma from \cite{Nikos-Host2}.
\begin{lem}\label{NH}Let $(a_n)$ be a bounded sequence of complex numbers. Then, we have
$$\Big|\frac{1}{\pi(N)}\sum_{\overset{p\textrm{~~prime}}{p \leq N}}a_p-\frac1{N}\sum_{n=1}^{N}\Lambda'(n)a_n\Big|\tend{N}{+\infty}0.$$
\end{lem}

\begin{proof}[\textbf{Proof of Theorem \ref{main-Mangoldt}.}] For the case $k=1$. The result holds for any dynamical system. Indeed,
the result follows from Lemma \ref{NH} combined with the main result
from \cite{Wierdl}. Let us assume from now that $k \geq 2$. Then,
as in Ford-Green-Konyagin-Tao's proof, it suffices to see that the following holds
$$\frac1{N^k}\sum_{\vec{n}\in [1,N]^k}\Big(\prod_{e \in C^*}\Lambda'_{b_i,W}(\vec{n}.e)-1\Big) \prod_{e \in C^*}f_e(T_e^{\vec{n}.e}x)), \eqno(NCSM)$$
where $b_i \in [1,W]$, $i=1,\cdots,k$ coprime to $W$.\\
But $\Lambda'_{b_i,W}-1$ is orthogonal to the nilsequences by Green-Tao result. Therefore the limit of the quantity $(CNSM)$ is zero.
\end{proof}
\begin{que}

According to our result Theorem \ref{main-Mangoldt}, we ask if is it true that the cubic nonconventional ergodic averages of any order with
von Mangoldt function weight converge almost surely?
\end{que}

\begin{thank}
The first author would like to express his heartfelt thanks to Professor Benjamin Weiss
for the discussions on the subject. It is a great pleasure also for him to acknowledge the warm
hospitality of University of Science and Technology of China and Fudan University where a part of this work has been done.
We also thank Frantzikinakis for very useful discussion on the subject.
\end{thank}


\end{document}